\documentclass[a4paper]{article}

\usepackage[english]{babel}

\usepackage[utf8]{inputenc}
\usepackage[utf8]{inputenc}
\usepackage{amsmath}
\usepackage{graphicx}
\usepackage{amssymb}
\usepackage{amsthm}
\usepackage{tikz-cd}
\usepackage{mathrsfs}
\usepackage[colorinlistoftodos]{todonotes}
\usepackage{enumitem}
\usepackage{yfonts}
\usepackage{ dsfont }
\usepackage{MnSymbol}

\usepackage[affil-it]{authblk}

\title{A note on singular Hermitian Yang-Mills connections}

\author{Yang Li 	\thanks{Y.L. is supported by the Engineering and Physical Sciences Research Council [EP/L015234/1], the EPSRC Centre for Doctoral Training in Geometry and Number Theory (The London School of Geometry and Number Theory), University College London. The author is also funded by Imperial College London for his PhD studies. }}

\makeatletter
\def\@maketitle{%
	\newpage
	\null
	\vskip 2em%
	\begin{center}%
		\let \footnote \thanks
		{\Large\bfseries \@title \par}%
		\vskip 1.5em%
		{\normalsize
			\lineskip .5em%
			\begin{tabular}[t]{c}%
				\@author
			\end{tabular}\par}%
		\vskip 1em%
		{\normalsize \@date}%
	\end{center}%
	\par
	\vskip 1.5em}
\makeatother

\newtheorem{thm}{Theorem}[section]
\newtheorem{lem}[thm]{Lemma}

\theoremstyle{definition}

\newtheorem{eg}[thm]{Example}

\newtheorem{cor}[thm]{Corollary}

\newtheorem*{rmk}{Remark}

\newtheorem*{Acknowledgement}{Acknowledgement}

\newcommand{\cf}{\emph{cf.} }

\newcommand{\C}{\mathbb{C}}
\newcommand{\Z}{\mathbb{Z}}

\newcommand{\Lap}{\Delta}

\begin{document}
\maketitle

\begin{abstract}
We give an example of a homogeneous reflexive sheaf over $\C^3$ which admits a non-conical Hermitian Yang-Mills connection. This is expected to model bubbling phenomenon along complex codimension 2 submanifolds when the Fueter section takes zero value.
\end{abstract}

The guiding wisdom in the study of Hermitian Yang-Mills (HYM) connections is the Hitchin-Kobayashi correspondence principle, namely \emph{metric properties of HYM connections translate into algebro-geometric properties of holomorphic vector bundles and vice versa}. For instance, the celebrated Donaldson-Uhlenbeck-Yau theorem \cite{Siu} states that over compact K\"ahler manifolds, a holomorphic vector bundle admits a unique HYM connection if and only if it is polystable. This has been extended to reflexive sheaves by Bando and Siu \cite{BandoSiu}, and by now there is a developed theory comparing the compactified moduli spaces of HYM connections versus the stable vector bundles \cite{Ben}. In a more local setting, recent works \cite{JacobWalpuski2}\cite{Chen} show that for a large class of reflexive sheaf singularities, the local tangent cones of the HYM connections can be read off algebro-geometrically from the Harder-Narasimhan filtration.

\begin{rmk}
In our terminology, a HYM connection $A$ has \emph{tangent cone connection} $A_\infty$ if some rescaling sequence of $A$  converges to $A_\infty$ smoothly \emph{possibly away from a real codimension 4 subset}.
\end{rmk}

The purpose of this paper is to provide a surprising example which shows that in the noncompact setting the holomorphic structre alone \emph{does not} need to capture everything about HYM connections:

\begin{thm}
There is a HYM connection on the homogeneous reflexive sheaf $\ker( \underline{\C^3}\xrightarrow{(x,y,z)}\underline{\C}   )$ over the Euclidean space $\C^3$ with locally finite $L^2$ curvature, whose tangent cone at infinity is flat.
\end{thm}

Complex geometrically, the Euler sequence shows that the reflexive sheaf is isomorphic over $\C^3\setminus \{0\}$ to the pullback of the cotangent bundle of $\mathbb{CP}^2$ via $\C^3\setminus \{0\}\to \mathbb{CP}^2$. Indeed, by \cite{JacobWalpuski2}\cite{Chen} the local tangent cone at the origin must be the pullback of the Levi-Civita connection on $\Omega_{\mathbb{P}^2  }$ up to twisting by a central $U(1)$-connection. The surprise is that the HYM metric is not conical in our example, and the complex geometry does not predict the tangent cone at infinity. Thus the roles of the tangent cone at infinity and the local tangent cone must be fundamentally asymmetrical. This has a similar flavour to the recent discovery of exotic Calabi-Yau metrics on $\C^n$ \cite{Li}\cite{Ronan}\cite{Gabor}.

Despite the appearance our main result does not violate the Price monotonicity formula, which states that $r^{4-\text{dim}} \int_{B(r)} |F|^2$ must be an increasing function in the radius $r$. Geometrically, the curvature has faster than quadratic decay in the generic region near infinity, resulting in a flat tangent cone at infinity, but has slower than quadratic decay close to the $z$-axis, transverse to which the HYM connection is modelled on scaled copies of the standard one-instanton. Thus on large spheres the curvature becomes concentrated in a very small solid angle, and when we take the tangent cone the $L^2$ curvature is lost in the limit.

Our strategy is to produce an ansatz from the \emph{monad construction}, and then use some fairly standard nonlinear existence machinery  to find a HYM connection asymptotic to the ansatz near infinity. The monad construction is motivated by studying a family of HYM connections bubbling along a complex curve $S$ inside a Calabi-Yau 3-fold \cite{DonaldsonSegal}. Near a local patch of $S$, the HYM connections restricted to the normal directions to $S$ are modelled on (framed) ASD instantons, whose variation along $S$ is governed by a holomorphic map (called the \emph{Fueter map}) from $S$ to the framed instanton moduli space. We are interested in the simplest case of instantons with rank 2 and charge 1, so the framed moduli space is $\C^2/\Z_2$. Now the most generic kind of zero for a holomorphic map $\C\to \C^2/\Z_2$ is  up to linear change of coordinates given by $z\mapsto (z^{1/2}, 0)$ to leading order. Via the ADHM construction, this particular Fueter map translates into a monad ansatz for the HYM connection. In our actual construction we take a regularized version of the monad ansatz with better curvature decay properties.

%After zooming into the neighbourhood of the zero by scaling coordinates, the ambient metric is essentially Euclidean. The most generic kind of zero for a holomorphic map $\C\to \C^2/\Z_2$  By simultaneously applying holomorphic changes of coordinates on $\C$ and $\C^2/\Z_2$, the Fueter map can be modified to the canonical form $z\mapsto (z^{1/2},0)$. 

\begin{Acknowledgement}
The author thanks his supervisor Simon Donaldson and co-supervisor Mark Haskins for their inspirations, the Simons Center for hospitality, and Aleksander Doan for very useful discussions.
\end{Acknowledgement}

\section{An ansatz from monad construction}

We begin with a general curvature formula for the cohomology of a monad over a complex manifold.

\begin{lem}\label{curvatureformulamonadlemma}
Consider a monad   $E_0\xrightarrow{\alpha} E_1\xrightarrow{\beta} E_2$, namely a complex of Hermitian holomorphic vector bundles with $\alpha$ injective fibrewise and $\beta$ surjective fibrewise. Let $E=\ker \beta/\text{coker}( \alpha )$ be the cohomology bundle. Then the curvature $F_E$ of the natural induced connection on $E$ satisfies
\[
\langle F_E s, s'\rangle= 
\langle F_{E_1} s, s'\rangle
- \langle (\beta \beta^\dag)^{-1} (\nabla \beta) s, (\nabla \beta)s' \rangle
- \langle (\alpha^\dag \alpha)^{-1} (\nabla \alpha^\dag) s, (\nabla \alpha^\dag)s' \rangle,
\]
where $F_{E_1}$ is the Chern connection on $E_1$, and $s,s'$ are representing smooth sections of $E$ satisfying $\alpha^\dag s= \alpha^\dag s'= \beta s= \beta s'=0$, and $\nabla \alpha^\dagger$, $\nabla \beta$ are covariant derivatives computed on the Hom bundles.
\end{lem}

\begin{eg}\label{ADHMconstruction}
(ADHM construction of  one-instantons) Start from the monads over Euclidean $\C^2_{x,y}$ 
\[
\underline{\C} \xrightarrow{ \alpha= (x, y, a_1, a_2)^t }\underline{\C^4} \xrightarrow{ \beta=(-y,x,b_1, b_2) }\underline{\C},
\]
where the underlines signify trivial vector bundles, and the parameters $a_1, a_2, b_1, b_2$ satisfy the ADHM equation
\[
a_1b_1+ a_2b_2=0, \quad |a_1|^2+ |a_2|^2= |b_1|^2+ |b_2|^2>0.
\]
The natural connections on cohomology bundles $E_{a,b}$ are ASD instantons on $\C^2$ with rank 2, charge 1 and curvature scale $\sim \sqrt{ |a_1|^2+|a_2|^2}$. The situation with $a_1=a_2=b_1=b_2=0$ is viewed as a degenerate case.
As Hermitian vector bundles $E_{a,b}$ are identified as $\ker \beta \cap \ker \alpha^\dagger\subset \underline{\C^4}=\underline{\C^2}\oplus \underline{\C^2} $, and the monad construction provides a natural projection map into the second $\underline{\C^2} $ factor, giving a trivialisation of $E_{a,b}$ near infinity known as \emph{framing data}. Notice the framed instantons are isomorphic under the $U(1)$-symmetry
\[
(a_1,a_2,b_1,b_2)\mapsto (a_1 e^{i\theta}, a_2 e^{i\theta}, b_1 e^{-i\theta},b_2 e^{-i\theta} ).
\]
The \emph{moduli space of framed instantons} centred at the origin is 
\[
\{ a_1b_1+a_2b_2=0, |a_1|^2+ |a_2|^2= |b_1|^2+ |b_2|^2     \}/U(1) \simeq \{ a_1b_1+a_2b_2=0 \}/\C^* \simeq \C^2/\Z_2.
\]
\end{eg}

We now describe our \emph{main ansatz}. Take the monad over Euclidean $\C^3_{x,y,z}$
\begin{equation}\label{monad}
\underline{\C}  \xrightarrow{ \alpha=(x, y, 1, 0)^t }\underline{\C^4} \xrightarrow{ \beta=(-y,x,0, z) }\underline{\C},
\end{equation}
and equip the trivial bundle $\underline{\C^4}$ with
 the \emph{nonstandard Hermitian structure}  given by the diagonal matrix $h_{\C^4}=\text{diag}( (|\vec{x}|^2+1  )^{-1/2}, (|\vec{x}|^2+1  )^{-1/2},   1,1 )$ where $|\vec{x}|^2=|x|^2+|y|^2+|z|^2$. The monad has a unique singular point at the origin in $\C^3$ where $\beta$ fails to be surjective; in fact simple linear algebra shows the cohomology sheaf $E$ is isomorphic to the homogeneous coherent sheaf $\ker ( \underline{\C^3} \xrightarrow{ (x,y,z)} \underline{\C}   )$, so is in particular \emph{reflexive}. Using the Euler sequence over $\mathbb{P}^2$
\[
0\to \Omega_{\mathbb{P}^2 } \to \mathcal{O}(-1)^{\oplus 3} \to  \mathcal{O} \to 0,
\]
$E$ is isomorphic as vector bundles over $\C^3\setminus \{0\}$ to the pullback of $\Omega_{\mathbb{P}^2} $ via $\C^3\setminus \{0\}\to\mathbb{P}^2 $.

The key point is that the mean curvature on $E$ has fast enough decay at infinity.

\begin{lem}\label{meancurvaturebound}
The curvature $F_E$ of the natural connection on $E$ admits the estimate
$
|\Lambda F_E| \leq C \ell, 
$
where $\ell$ is a fixed function uniformly equivalent to
\begin{equation}\label{elldefinition}
\ell \sim \begin{cases}
 \frac{  1   }{ (|x|^2+|y|^2+|z|)|\vec{x}| }  , \quad  &|\vec{x}|\gtrsim 1,
\\
\frac{ 1 }{ |x|^2+|y|^2+|z|^2   }   , \quad & |\vec{x}|\lesssim 1.
\end{cases}
\end{equation}
\end{lem}

\begin{proof}
We  shall compute the curvature $F_E$. Taking into account the nonstandard Hermitian structure, the adjoint maps are given by
\[
\alpha^\dagger=( \bar{x}( |\vec{x}|^2+1)^{-1/2}, \bar{y}( |\vec{x}|^2+1)^{-1/2} , 1, 0  ), \quad \beta^\dagger= (-\bar{y}\sqrt{ |\vec{x}|^2+1}, \bar{x}\sqrt{ |\vec{x}|^2+1}, 0,  \bar{z}  )^t,
\]
hence 
\begin{equation}\label{betabetadagger}
\alpha^\dag \alpha= (|x|^2+|y|^2)( |\vec{x}|^2+1)^{-1/2}+1  , \quad \beta\beta^\dagger= (|x|^2+|y|^2)\sqrt{ |\vec{x}|^2+1}    +  |z|^2 .
\end{equation}
Since $\alpha, \beta$ are holomorphic, $\alpha^\dagger, \beta^\dag $ are antiholomorphic, 
\begin{equation}\label{nablaalphabeta}
\begin{cases}
\nabla \alpha^\dagger=\bar{\partial} \alpha^\dagger= ( d\bar{x}( |\vec{x}|^2+1)^{-1/2}- \frac{  \bar{x}(xd\bar{x}+ yd\bar{y}+ zd\bar{z}) }{ 2 ( |\vec{x}|^2+1)^{3/2}  } , d\bar{y}( |\vec{x}|^2+1)^{-1/2}- \frac{  \bar{y}(xd\bar{x}+ yd\bar{y}+ zd\bar{z}) }{ 2 ( |\vec{x}|^2+1)^{3/2}  }, 0, 0),
\\
\nabla \beta= ( \nabla \beta^\dag )^\dagger= ( \bar{\partial} \beta^\dag  )^\dagger= ( -dy - \frac{  y(\bar{x}dx + \bar{y}dy+ \bar{z}dz) }{ 2 ( |\vec{x}|^2+1)  }   , dx+ \frac{  x(\bar{x}dx + \bar{y}dy+ \bar{z}dz) }{ 2 ( |\vec{x}|^2+1)  } , 0, dz    ).
\end{cases}
\end{equation}

Let $s=(s_1, s_2,s_3,s_4)^t$ be a smooth local section of $E$ represented as a section of $\underline{\C^4}$ with $\beta s= \alpha^\dag s=0$, so by expressing $s_1, s_2$ in terms of $s_3, s_4$,
\[
|s_1|+|s_2| \leq \frac{ (|s_3|+|s_4|) (|\vec{x}|+1) }{ |x|+|y|  } .
\]
Under the Hermitian structure
\[
|s|_h^2= (|s_1|^2+ |s_2|^2 )(|\vec{x}|^2+1  )^{-1/2} +  (|s_3|^2+|s_4|^2),
\]
we have
\begin{equation}\label{s1+s2}
|s_1|+|s_2|\leq C\min\{ (|\vec{x}|+1  )^{1/2}, \frac{  |\vec{x}|+1 }{ |x|+|y|  } \} |s|_h.
\end{equation}

%\[
%\begin{cases}
%(\nabla \alpha^\dagger)s=(s_1 d\bar{x}+ s_2 d\bar{y}) ( |\vec{x}|^2+1)^{-1/2} + O( \frac{ (|s_1|+|s_2|)(|x|+|y|) } {  |\vec{x}|^2+1  }   ),
%\\
%(\nabla \beta) s= -s_1dy+ s_2dx+O( \frac{ (|s_1|+|s_2|)(|x|+|y|) } {  |\vec{x}|^2+1  }   ).
%\end{cases}
%\]

Now the Chern curvature on $\underline{\C^4}$ is given by $F_{E_1}= \bar{\partial} ( \partial h_{\C^4} h_{\C^4}^{-1} ),
$
so
\[
\langle \sqrt{-1} F_{E_1} s, s \rangle = O( \frac{|s_1|^2+|s_2|^2  }{  ( |\vec{x}|^2+1  )^{3/2}   }   ).
\]
Substituting (\ref{betabetadagger})(\ref{nablaalphabeta})(\ref{s1+s2}) into the curvature formula in Lemma \ref{curvatureformulamonadlemma}, 
\begin{equation}\label{curvatureansatz}
\begin{split}
\langle \sqrt{-1} \Lambda F_E s,s\rangle =& -(\alpha^\dag \alpha)^{-1} (|\vec{x}|^2+1)^{-1} \sqrt{-1} \Lambda (s_1 d\bar{x}+s_2d\bar{y})\wedge (\bar{s}_1 d{x}+\bar{s}_2d{y})\\
&- (\beta \beta^\dag)^{-1} \sqrt{-1} \Lambda (-s_1 dy+s_2dx)\wedge (-\bar{s}_1 d\bar{y}+\bar{s}_2d\bar{x})+ O( \ell |s|_h^2  ).
%
%& +O( \frac{ |s_1|^2+|s_2|^2}{  (|\vec{x}|^2+1)^{3/2}  }       )+ O( (\beta\beta^\dag)^{-1} |s_4|^2  ) + O(  \frac{ (|x|+|y| )  ( |s_1|^2+|s_2|^2 )}{ (\beta\beta^\dag) (|\vec{x}|^2+1)^{1/2}  }    )  .
\end{split}
\end{equation}
Combined with a cancellation effect expressed by the inequality
\begin{equation}\label{cancellationeffect}
\begin{split}
(|\vec{x}|^2+1)^{-1}(\alpha^\dag \alpha)^{-1}- (\beta \beta^\dag)^{-1}= O( \frac{ 1}{ (\beta\beta^\dag) \sqrt{ |\vec{x}|^2+1}     }  ),
\end{split}
\end{equation}
this implies
$
\langle \sqrt{-1} \Lambda F_E s,s\rangle= O( \ell |s|_h^2   ) ,      
$
or equivalently $|\Lambda F_E|=O(\ell)$ as required.
\end{proof}

\begin{rmk}\label{tangentconeatinfinity}
From the proof one can readily extract estimates on $|F_E|$. In particular near the origin $|F_E|=O(|\vec{x}|^{-2})$ is \emph{locally in $L^2$}. For $|x|+|y|\gtrsim |z|^{1/2}+1$, the curvature decays like $|F_E|=O( \frac{ |\vec{x}|}{   (|x|^2+|y|^2 )^2}  )$. In particular when $|x|+|y|\gtrsim |\vec{x}|\gg 1$, the curvature has cubic decay, so \emph{the tangent cone at infinity is flat}. The curvature becomes concentrated near the $z$-axis near spatial infinity.
\end{rmk}

For later usage, we estimate the potential integral
\begin{equation}
G(x,y,z)= \int_{\C^3} \frac{\ell(\vec{x}')}{ |(x-x',y-y',z-z')|^4  }  d\text{Vol}(\vec{x}').
\end{equation}

\begin{lem}\label{barrier}
The function $G>0$ is well defined on $\C^3\setminus \{0\}$, and satisfies
\[
\begin{cases}
\Lap G= \text{const}\cdot \ell, \\
 |G|\leq C |\vec{x}|^{-1} \max(  \log  \frac{|\vec{x}|}{ |x|+|y|+|z|^{1/2}  }, 1 ) , \quad |\vec{x}|\gtrsim 1.
\end{cases}
\]
\end{lem}

\begin{proof}
The singularity of the source $\ell$ at the origin is $O( \frac{1}{ |\vec{x}|^2 }   )$, which is mild enough to guarantee the potential integral is well defined and has the correct Laplacian. We focus on the estimates for $|\vec{x}|\gtrsim 1$. The $L^1$ integral of the source $\ell$ at a scale $|\vec{x}'|\sim 2^k$ is bounded by
\[
\int_{ |\vec{x}'|\sim 2^k} \frac{ 1}{ (|x'|^2+|y'|^2) |\vec{x}'|    } d\text{Vol}(\vec{x}')\lesssim 2^{3k},
\]
so by summing over all dyadic scales,
\[
|G- \int_{|\vec{x}|\sim |\vec{x}'|} \frac{\ell(\vec{x}')  d\text{Vol}(\vec{x}')  }{ |(x-x',y-y',z-z')|^4  }  | \lesssim |\vec{x}|^{-1}.
\]
Thus the only important contribution comes from the dyadic scale $|\vec{x}|\sim |\vec{x}'|$, %which is bounded by
%\[
%\begin{split}
%|\vec{x}|^{-1}\int_{|\vec{x}|\sim |\vec{x}'|} \frac{ \ell(\vec{x}')       1}{ |(x-x',y-y',z-z')|^4  }  \frac{d\text{Vol}(\vec{x}')}{ |x'|^2+|y'|^2+|z'|   }
%\end{split}
%\]
which for $|x|^2+|y|^2\gtrsim |z|$ is controlled by
\[
\int_{|\vec{x}|\sim |\vec{x}'|} \frac{1}{ |(x-x',y-y',z-z')|^4  }  \frac{d\text{Vol}(\vec{x}')}{ (|x'|^2+|y'|^2)|\vec{x}|  } \lesssim  |\vec{x}|^{-1  }  \max(1, \log \frac{|\vec{x}|}{ |x|+|y|  }   ),
\]
and for $|x|^2+|y|^2\lesssim |z|$ is controlled by $O( |\vec{x}|^{-1} | \log   \frac{|\vec{x}| }{|z|^{1/2}}    | )$. Combining the above shows the claim.
\end{proof}

\subsection{Asymptotic geometry near infinity}

%Now we improve the asymptotic geometry in Remark \ref{tangentconeatinfinity} to higher order. %The basic picture is that the curvature decays slowly near the complex line $\{x=y=0\}$ with regularity scale $O(|z|^{1/2})$, but in the generic region it decays faster than quadratically and the connection is almost flat.

First we examine the asymptotic geometry for $|x|+|y|\gg |z|^{1/2}+1$, namely the generic region near spatial infinity. Consider the case $|x|\lesssim |y|$, so $|y|\gg |\vec{x}|^{1/2}+1$. A basis of holomorphic sections on $E$ can be represented by sections of $\ker \beta$:
\[
s_{(1)}= (0,0,1,0)^t, \quad s_{(2)}= (z/y, 0, 0, 1)^t.
\]
The projections of $s_1,s_2$ to the orthogonal complement of $\text{Im}(\alpha)$ are respectively
\[
\begin{cases}
s_1'= s_{(1)}- \alpha (\alpha^\dag \alpha)^{-1} \alpha^\dag s_{(1)} 
=s_{(1)}- \frac{1}{ (|x|^2+|y|^2)(|\vec{x}|^2+1)^{-1/2} +1        } (x,y,1,0)^t,
\\ 
s_2'= s_{(2)}- \alpha (\alpha^\dag \alpha)^{-1} \alpha^\dag s_{(2)}= s_{(2)}- \frac{\bar{x} z/y}{ |x|^2+|y|^2 +  (|\vec{x}|^2+1)^{1/2}        } (x,y,1,0)^t.
\end{cases}
\]
The Hermitian metric $H_0$ on the cohomology bundle is represented by the matrix
\[
H_0(s_{(i)},s_{(j)})= h_{\C^4}( s_i', s_j'  )= \delta_{ij}+ O( \frac{ |\vec{x}|}{ |y|^2 }   ).
\]
It follows from elementary Taylor expansions that
\[
|\partial^k H_0 |=  O( \frac{ |\vec{x}|}{ |y|^{2+k} }   ), \quad k\geq 1,
\]
where $\partial^k$ refers to the $k$-th partial derivatives. The natural connection $\nabla_E$ on $E$ is just the Chern connection induced by the Hermitian structure $H_0$, so in particular 
$
|F_E|= O( \frac{|\vec{x}|}{ |y|^4 }   )
$ compatible with Remark \ref{tangentconeatinfinity}. The mean curvature $\Lambda F_E$ has better decay properties. For this, we can apply the more accurate formula in Lemma \ref{curvatureformulamonadlemma} to derive an explict expression for the  matrix $( \langle \Lambda F_E s_i', s_j'\rangle )$ as formula (\ref{curvatureansatz}). The same cancellation effect as in (\ref{cancellationeffect}) happens. The higher order version of Lemma \ref{meancurvaturebound} then follows from Taylor expansion considerations:
\[
|\partial^k \langle \Lambda F_E s_i', s_j'\rangle |\leq C(k) |\vec{x}|^{-1} |y|^{-2-k}       ,
\]
or equivalently $|\nabla^k_E (\Lambda F_E) |\leq C(k) |\vec{x}|^{-1}  |y|^{-2-k} $.

Similarly in the case $|x|\gtrsim |y|$, we can find another basis of holomorphic sections on $E$, with
\[
H_0=\delta_{ij}+ O(\frac{ |\vec{x}|}{ |x|^{2} }    ), \quad |\partial^k H_0 |\leq C(k) \frac{ |\vec{x}|}{ |x|^{2+k} }   , \quad k\geq 1,
\]
and $|\nabla_E^k (\Lambda F_E) |\leq C(k) |\vec{x}|^{-1}|x|^{-2-k}  $. To summarize, the Hermitian structure on $E$ in the generic region $|x|+|y|\gg |z|^{1/2}+1$ is \emph{approximately flat}.

%\begin{equation}
%|\nabla_E^k  F_E|= O( \frac{ |\vec{x}|}{ (|x|+|y|)^{4+k} }   ), \quad |\nabla_E^k (\Lambda F_E)|= O( \frac{ |\vec{x}|}{ (|x|+|y|)^{4+k} }   ) \quad k\geq 0.
%\end{equation}
% with regularity scale $\rho\sim \frac{ |x|^2+|y|^2}{ |\vec{x}|^{1/2}  }  $.

Next we turn to the vicinity of the $z$-axis  $1\ll |z|^{1/2}\lesssim |x|+|y|$. Observe that if the ambient Hermitian structure on $\underline{\C^4}$ is changed from $h_{\C^4}$ to 
\[
 (|\vec{x}|^2+1)^{1/2} h_{\C^4}= \text{diag}( 1, 1,  (|\vec{x}|^2+1)^{1/2}, (|\vec{x}|^2+1)^{1/2}    ),
\]
then the induced connection on $E$ is \emph{twisted by a $U(1)$ connection} with curvature $\frac{1}{2}\bar{\partial} \partial \log (|\vec{x}|^2+1  )$. We shall focus on this twisted situation around a given point $(0,0, \zeta)\in \C^3$ with $|\zeta|\gg 1$, and choose a square root $\zeta^{1/2}$. After rescaling the basis vectors on $\underline{\C^4}$, the twisted monad can be written as
\begin{equation}\label{monadnearbubblelocus}
\C\xrightarrow{  ( x,y, \zeta^{1/2},0     )^t      } \underline{\C^4} \xrightarrow{ (-y,x, 0, \zeta^{1/2}) }\underline{\C},
\end{equation}
where the Hermitian structure on $\underline{\C^4}$ is $\tilde{h}_{ \C^4}=\text{diag}( 1, 1,  \frac{(|\vec{x}|^2+1)^{1/2}}{|\zeta| }, \frac{ |\zeta|  (|\vec{x}|^2+1)^{1/2} }{ |z|^2 }    )$. For $|x|+|y|+|z-\zeta|\lesssim |\zeta|^{1/2}$, Taylor expansion shows
\[
\begin{cases}
\tilde{h}_{ \C^4}= \text{diag}(1,1,1+O( \frac{ 1  +|z-\zeta| }{ |\zeta| }  ),1+O( \frac{ 1+|z-\zeta|}{ |\zeta| }  )   ),
\\
|\partial^k \partial_x \tilde{h}_{ \C^4}|+ |\partial^k \partial_y \tilde{h}_{ \C^4}|  \leq C(k)\frac{ 1}{ |\zeta|^{3/2+k/2} }  , \quad k\geq 0.
\\
|\partial^k \partial_z \tilde{h}_{ \C^4} |\leq C(k)\frac{1}{ |\zeta|^{1+k/2} }   ,  \quad  |\partial^k \partial_z \partial_{\bar{z}} \tilde{h}_{ \C^4} |\leq C(k)\frac{1}{ |\zeta|^{2+k/2} }, \quad k\geq 0.
\end{cases}
\]
To leading order, the ambient Hermitian metric on $\underline{\C^4}$ is Euclidean, and the twisted monad (\ref{monadnearbubblelocus}) dimensionally reduces to the monad in the  ADHM construction with parameters $(a_1,a_2,b_1,b_2)=(\zeta^{1/2}, 0, 0, \zeta^{1/2})$ (\cf Example \ref{ADHMconstruction}). In particular, the connection on the cohomology bundle of the twisted monad is approximately a framed instanton whose moduli parameter is identified as $(\zeta^{1/2},0)\in \C^2/\Z_2$.

\begin{rmk}
From a more global viewpoint, the twisted monad connection in the normal direction to the $z$-axis is described by the \emph{Fueter map} into the moduli space of framed instantons:
\begin{equation}
\C\to \C^2/\Z_2, \quad \zeta\mapsto (\zeta^{1/2}, 0).
\end{equation}
Notice the Fueter map is independent of the choice of square root $\pm \zeta^{1/2}$. As mentioned in the introduction, our monad construction (\ref{monad}) is best thought as a regularized version of the monad induced by this Fueter map. The extra issue of twisting by a $U(1)$ connection is required for the almost flatness of the connection far away from the $z$-axis.
\end{rmk}

One can then estimate the difference between $\nabla_E$ and the instanton connection $\nabla_\zeta$ associated with the ADHM monad, in the region $\{  |x|+|y|\lesssim |\zeta|^{1/2}, z=\zeta    \}$, by taking into account both the deviation of $\tilde{h}_{\C^4}$ from being Euclidean, and the effect of $U(1)$-twisting:
%\[
%| \nabla_\zeta^k ( \nabla_E- \nabla_\zeta )|= O( |z|^{-1-k/2}  ), \quad k\geq 0,
%\]
%and in particular
\[
|\nabla_\zeta^k (F_E- F_{\nabla_\zeta} )|\leq C(k) |z|^{-2-k/2}      , \quad | \nabla_E^k (\Lambda F_E)|\leq C(k) |z|^{-2-k/2}   ,
\]
and the regularity scale of $\nabla_E$ in this region is comparable to the regularity scale of $\nabla_\zeta$ which is $\sim |\zeta|^{1/2} $.

Combining the above, we have a unified higher order estimate for $\Lambda F$:

\begin{cor}\label{meancurvaturehigherorder}
In the region $|x|\gtrsim 1$, 
\[
|\nabla^k_E (\Lambda F_E) |= O( |\vec{x}|^{-1} (|x|+ |y|+|z|^{1/2} )^{-2-k}  ).
\]
\end{cor}

\section{Perturbation into HYM metric}

We seek a nonlinear perturbation of the ansatz to a genuine HYM connection. The strategy is to solve Dirichlet boundary value problems on larger and larger domains exhausting $\C^3$, obtain uniform estimates and then extract limits as in \cite{Ni}. The analysis involved is by now fairly standard.

%We begin by stating an existence theorem for singular Hermitian Yang-Mills connections over domains, which amounts to a simple combination of Donaldson's solution to the Dirichlet problem in the smooth case \cite{Donaldsonboundary}, and Bando and Siu's existence result for reflexive sheaves \cite{BandoSiu}.

\begin{thm}
Let $E$ be a reflexive sheaf over a compact K\"ahler manifold $(\overline{Z}, \omega)$ with nonempty boundary $\partial Z$, which is locally free near the boundary. For any Hermitian metric $f$ on the restriction of $E$ to $\partial Z$ there is a unique Hermitian $H$ on $E$, which is smooth on the locally free locus, has finite $L^2$ curvature, and
\[
\sqrt{-1}\Lambda F_H=0 \text{ in } Z, \quad H=f \text{ over } \partial Z.
\]
\end{thm}

\begin{proof}
(Sketch) Donaldson \cite{Donaldsonboundary} proved the special case when $E$ is a vector bundle using the heat flow method. The key point is that $|\Lambda F|$ is a subsolution to the heat equation with zero boundary data, which forces $|\Lambda F|$ to decay exponentially in time. This together with uniform $C^k$ control on the connection in the flow leads to long time existence and convergence to HYM connection at infinite time. Here the issue of stability does not appear.

The singularity problem can be addressed by the continuity method due to Bando and Siu \cite{BandoSiu};  we follow the exposition in \cite{JacobWalpuksi}, Section 6. The idea is to take repeated blow ups $\tilde{Z}$ in the interior of $\bar{Z}$ so that $E|_{ \bar{Z}\setminus \text{Sing}(E)  }$ extends as a vector bundle $\tilde{E}$ across the exceptional locus. Equip $\tilde{E}$ with a fixed reference Hermitian metric $\tilde{H}_1$ admitting the given boundary data. One can find a sequence of degenerating K\"ahler metrics $\omega_\epsilon\to \omega$ on $\tilde{Z}$ as $\epsilon \to 0$, such that
\begin{itemize}
\item $\omega_\epsilon$ agrees with $\omega$ on the $O(\epsilon)$ neighbourhood of the exceptional locus.
\item Near the blow up loci $\omega_\epsilon$ is locally modelled on a rescaling of the standard metric on $\text{Bl}_0 \C^k \times \C^{n-k}$.
\item The curvature of $\tilde{H}_1$ has uniform $L^2$ curvature bound as $\epsilon\to 0$.
\item The metrics $\omega_\epsilon$ have uniform Dirichlet-Sobolev constants.
\end{itemize}

Applying Donaldson's result to solve the Dirichlet problem, we obtain HYM connections $\tilde{H}_\epsilon$ on $\tilde{E}$ for the background metrics $\omega_\epsilon$. Using the almost subharmonicity estimate (in the analyst's Laplacian convention)
\[
\Lap \text{Tr} \log (\tilde{H}_\epsilon \tilde{H}_1^{-1}  ) \geq - |\Lambda F_{ \tilde{H}_1 }|, \quad  \Lap \text{Tr} \log (\tilde{H}_1 \tilde{H}_\epsilon^{-1}  ) \geq - |\Lambda F_{ \tilde{H}_1 }|,
\]
and the uniform Dirichlet-Sobolev inequality, there is a uniform $L^2$ bound on $\text{Tr}\log (\tilde{H}_\epsilon \tilde{H}_1^{-1}  )$ and $\text{Tr} \log (\tilde{H}_1 \tilde{H}_\epsilon^{-1}  )$. On any compact subset of the locally free locus of $E$, the almost subharmonicty  implies furthermore $L^\infty$ estimates on $\text{Tr}\log (\tilde{H}_\epsilon \tilde{H}_1^{-1}  )$ and $\text{Tr} \log (\tilde{H}_1 \tilde{H}_\epsilon^{-1}  )$, so that $\tilde{H}_\epsilon$ is locally uniformly equivalent to $\tilde{H}_1$ independent of $\epsilon$. Then the Bando-Siu interior estimate (\cf Appendix C, D in \cite{JacobWalpuksi}) gives $C^k_{loc}$-estimates on $\tilde{H}_\epsilon$ over the locally free locus of $E$, uniform in $\epsilon$. Furthermore, there are uniform $L^2$ curvature bounds on $\tilde{H}_\epsilon$ because of the topological energy formula for HYM connections
\begin{equation}\label{topologicalenergy}
\int_{\bar{Z}} |F_{ \tilde{H}_\epsilon }|^2 \omega_\epsilon^n= -\text{const} \cdot \int_{\bar{Z}} \text{Tr}(F_{ \tilde{H}_\epsilon} \wedge  F_{ \tilde{H}_\epsilon } ) \omega_\epsilon^{n-2}=- \text{const} \int_{\partial Z} \text{Tr} ( \partial \tilde{H}_\epsilon  \tilde{H}_\epsilon^{-1} \wedge F_{ \tilde{H}_\epsilon }   ) \wedge \omega_\epsilon^{n-2}.
\end{equation}

Now taking a subsequential weak limit as $\epsilon\to 0$, we obtain a HYM metric $H$ over the locally free locus of $E$, which must have bounded $L^2$ curvature. The uniqueness statement follows from the fact that if $H$, $H'$ are two solutions to the Dirichlet problems with finite $L^2$-curvature, then $\text{Tr} \log (H H'^{-1}  )$ and $\text{Tr} \log (H' H^{-1}  )$ are both subharmonic.
\end{proof}

\begin{thm}
There is a HYM connection $H$ on $E$ over Euclidean $\C^3$ with locally finite $L^2$ curvature, which admits the decay estimates on $\{|x|\gtrsim 1\}$:
\begin{equation}\label{asymptoticestimate}
|\nabla^k_E  \log ( H H_0^{-1}   ) |\leq C(k)( |x|+|y|+|z|^{1/2}    )^{-k}|\vec{x}|^{-1} \max(1, \log  \frac{|\vec{x}}{|x|+|y|+|z|^{1/2}}) , \quad k\geq 0.
\end{equation}
In particular, the tangent cone at infinity is flat. The asymptotic estimate (\ref{asymptoticestimate}) and the HYM condition $\Lambda F_H=0$ determine $H$ uniquely.
\end{thm}

\begin{proof}
Recall $H_0$ is the natural Hermitian metric for the cohomology sheaf $E$ of our monad (\ref{monad}), whose curvature is $F_E$. 
We solve the Dirichlet problem on large balls $B(R)\subset \C^3$ with boundary data $H_0$ on $\partial B(R)$, and denote the solution as $H_R$. Crucially we need the almost subharmonicity estimate  (\cf Lemma 2.5 in \cite{Ni}, and Lemma \ref{meancurvaturebound}):
\begin{equation}\label{subharmonicity}
\begin{cases}
\Lap \log  \text{Tr} (H_R H_0^{-1}) \geq  -|\Lambda F_E|\geq - C\ell  ,
\\
\Lap \log \text{Tr} (H_0 H_R^{-1}) \geq - C\ell ,
\end{cases}
\end{equation}
Notice near the origin (\ref{subharmonicity}) continues to hold in the distributional sense, since the $L^2$ curvature of $H_0$ and $H_R$ are finite, and the singularity has complex codimension 3 (\cf proof of Proposition 3.1 in \cite{JacobWalpuski2}). Using the boundary condition $\text{Tr} (H_R H_0^{-1})= \text{Tr} (H_0 H_R^{-1})=\text{rank}(E)=2$, we apply 
Lemma \ref{barrier} and the comparison principle to get
\[
\log  \frac{\text{Tr} (H_R H_0^{-1})}{2} \geq -CG, \quad  \log \frac{\text{Tr} (H_0 H_R^{-1})}{2} \geq -CG,
\]
or equivalently there is a \emph{pointwise estimate} on $B(R)$:
\begin{equation}
e^{-CG} H_0 \leq H_R\leq e^{CG} H_0, 
\end{equation}
which is \emph{uniform} as $R\to \infty$. Using the upper bound for $|G|$ in $\{ |x|\gtrsim 1 \}$ provided by Lemma \ref{barrier},
\[
|\log (H_R H_0^{-1}) |\leq C |\vec{x}|^{-1}\max(1, \log \frac{|\vec{x}|}{|x|+|y|+|z|^{1/2}  } ).
\]

Recall also the higher order control on the mean curvature of the ansatz in Lemma \ref{meancurvaturehigherorder}. Applying Bando and Siu's interior regularity estimate (\cf Appendix C, D in \cite{JacobWalpuksi}) to rescaled balls, we derive the higher order estimate on $\{ 1\lesssim |\vec{x}|< R/2 \}$:
\[
|\nabla^k_E  \log ( H_R H_0^{-1}   ) |\leq C(k)   ( |x|+|y|+|z|^{1/2}    )^{-k} |\vec{x}|^{-1}\max(1, \log \frac{|\vec{x}|}{|x|+|y|+|z|^{1/2}  } )   , \quad k\geq 0.
\]
Using the topological energy formula similar to (\ref{topologicalenergy}), the $L^2$-curvature inside the unit ball is controlled. All estimates are uniform in $R$. Taking a subsequential limit as $R\to \infty$, we obtain the HYM connection $H$ with estimates (\ref{asymptoticestimate}). Since the deviation is so small that the leading order asymptotic geometry at infinity is unchanged, and in particular $|F_H|=O( \frac{ |\vec{x}|}{ (|x|^2+|y|^2+|z|)^2   }  )$ as $|\vec{x}|\to \infty$, we see the tangent cone at infinity is flat.

To see the uniqueness, notice if $H'$ is another HYM metric satifying (\ref{asymptoticestimate}), then both $\log \frac{ \text{Tr}  (H'H^{-1})}{2}$ and $\log \frac{\text{Tr} (HH'^{-1}) }{2}$ are subharmonic in the distributional sense, and are asymptotic to zero at infinity.
\end{proof}

The \emph{tangent cone connection at the origin} is determined a priori by complex geometry of the reflexive sheaf $E\simeq \ker(\C^3\xrightarrow{(x,y,z)} \C)$ (\cf \cite{Chen}\cite{JacobWalpuski2}). Complex geometrically, it is isomorphic to $E$, or equivalently the pullback of the cotangent bundle on $\mathbb{CP}^2$. The Levi-Civita connection on $\Omega_{\mathbb{P}^2}$ is HYM with respect to the integral Fubini-Study metric $\omega_{FS}$:
\[
\Lambda_{\omega_{FS}} F_{\Omega_{\mathbb{P}^2} }= 4\pi \mu \cdot \text{Id}_{ \Omega_{\mathbb{P}^2}}  , \quad \mu= \frac{\text{degree} }{\text{rank}  }(\Omega_{\mathbb{P}^2} )= -\frac{3}{2}.
\]
The tangent cone at the origin is the pullback of the Levi-Civita connection, up to a conformal change of the Hermitian structure by a factor $|\vec{x}|^{-3}$ which cancels out the Einstein constant.
Equivalently, the tangent cone is the natural connection on the kernel of 
$
\underline{\C^3}\xrightarrow{(x,y,z)} \underline{\C},
$
but $\underline{\C^3}$ is equipped with the \emph{nonstandard Hermitian structure} $\text{diag}( \frac{1}{|\vec{x}|}, \frac{1}{|\vec{x}|}, \frac{1}{|\vec{x}|}    )$. An application of \cite{JacobWalpuski2} shows that our HYM connection on $E$ is asymptotic to the tangent cone at the origin with polynomial decay rate.

%The pullback of the Levi-Civita connection to $\C^3\setminus \{0\}$ is HYM with respect to the standard Euclidean metric $\omega= \frac{\sqrt{-1}}{2}(dx\wedge d\bar{x}+dy\wedge d\bar{y}+dz\wedge d\bar{z})$, with Einstein constant

Some additional insights can be gained by studying the \emph{growth rate of holomorphic sections} as in \cite{Chen}. On a reflexive sheaf with a conical HYM connection, there is a convexity estimate (\cf Proposition 3.5 in \cite{Chen})
\[
(\int_{B_{1/4}} |s|^2)( \int_{B(1)} |s|^2   ) \geq ( \int_{B(1/2)} |s|^2   )^2. 
\]
A basic heuristic in \cite{Chen} is that
if a singular HYM connection is sufficiently close to being conical on a certain scale, then the convexity behaviour transfers to the HYM connection, so that $\log \int_{B(r) } |s|^2/ \log r$ has some monotonicity property, and one can define the \emph{local growth degree}
\[
d(s)= \frac{1}{2} \lim_{r\to 0} \frac{ \log \int_{B(r) } |s|^2 }{\log r  } -\dim_\C
\]
which induces a \emph{filtration on the germ of holomorphic sections},  intimately related to a Harder-Narasimhan-Seshardri filtration.

In the setting of our example, this motivates us to define the growth degree at infinity
\[
d_\infty (s)= \frac{1}{2} \lim_{r\to \infty} \frac{ \log \int_{B(r) } |s|^2 }{\log r  } -\dim_\C= \frac{1}{2}  \lim_{r\to \infty} \frac{ \log \int_{B(r) } |s|^2 }{\log r  } -3 ,
\] 
inducing a filtration on the holomorphic sections with \emph{at most polynomial growth}. It is then easy to check explicitly that there are holomorphic sections whose growth degree at infinity is larger than the growth degree at the origin, so \emph{the two filtration structures are different}. We expect that there are large classes of examples generalizing our construction, and these filtration structures should play a major role in a more systematic theory. In particular, it would be interesting to relate the filtration and the tangent cone at infinity.

\end{document}